\documentclass{amsart}
\usepackage[utf8x]{inputenc}
\usepackage{amsmath}
\usepackage{amssymb}
\usepackage{amsthm}
\usepackage{dsfont}
\usepackage{hyperref}
\usepackage[all]{xy}
\usepackage{color}

\newtheorem{theorem}{Theorem}[section]
\newtheorem{lemma}[theorem]{Lemma}
\newtheorem{proposition}[theorem]{Proposition}

\newtheorem{corollary}[theorem]{Corollary}
\newtheorem*{claim}{Claim}

\theoremstyle{definition}
\newtheorem{definition}[theorem]{Definition}

\theoremstyle{remark}
\newtheorem*{remark}{Remark}
\newtheorem*{note}{Note}

\newcommand{\pres}[2]{\langle \, #1 \mid #2 \, \rangle}
\DeclareMathOperator{\COF}{COF}

\title{Describing groups}
\author{Meng-Che ``Turbo" Ho}
\date{\today}

\begin{document}

\begin{abstract}
We study two complexity notions of groups -- the syntactic complexity of a computable Scott sentence and the $m$-degree of the index set of a group. Finding the exact complexity of one of them usually involves finding the complexity of the other, but this is not always the case. Knight et al.~\cite{Ca06}, \cite{Ca12}, \cite{Kn} determined the complexity of index sets of various structures.

In this paper, we focus on finding the complexity of computable Scott sentences and index sets of various groups. We give computable Scott sentences for various different groups, including nilpotent groups, polycyclic groups, certain solvable groups, and certain subgroups of $\mathbb{Q}$. In some of these cases, we also show that the sentence we give are optimal. In the last section, we also show that d-$\Sigma_2\subsetneq\Delta_3$ in the complexity hierarchy of pseudo-Scott sentences, contrasting the result saying d-$\boldsymbol{\Sigma}_2=\boldsymbol{\Delta}_3$ in the complexity hierarchy of Scott sentences, which is related to the boldface Borel hierarchy.
\end{abstract}

\maketitle

\section{Introduction}
One important aspect of computable structure theory is the study of the computability-theoretic complexity of structures. Historically, there are many natural questions of this flavor even outside the realm of logic. For example, the word problem for groups asks: for a given finitely-generated group, is there an algorithm that can determine if two words are the same in the group? It was shown in \cite{Ra60} that there is such an algorithm if and only if the group is computable in the sense of computability theory.

In this work, we will study the computability-theoretic complexity of groups. Among many different notions of complexities of a structure, we look at the quantifier complexity of a computable Scott sentence and the complexity of the index set.

\subsection{Background in recursive structure theory}
Instead of just using the first-order language, we will work in $\mathcal{L}_{\omega_1,\omega}$. This is the language where we allow countable disjunctions and countable conjunctions in addition to the usual first-order language. An classic theorem of Scott shows that this gives all the expressive power one needs for countable structures.

\begin{theorem}[Scott, \cite{Sc65}]
Let $L$ be a countable language, and $\mathcal{A}$ be a countable structure in $L$. Then there is a sentence in $\mathcal{L}_{\omega_1,\omega}$ whose countable models are exactly the isomorphic copies of $\mathcal{A}$. Such a sentence is called a \emph{Scott sentence} for $\mathcal{A}$.
\end{theorem}

To work in a computability setting, this is not good enough, because we also want the sentence to be computable in the following way:

\begin{definition}
We say a set is \emph{computably enumerable} (\emph{c.e.}, or \emph{recursively enumerable}, \emph{r.e.}) if there is an algorithm that enumerates the elements of the set.

We say a sentence (or formula) in $\mathcal{L}_{\omega_1,\omega}$ is \emph{computable} if all the infinite conjunctions and disjunctions in it are over c.e.~sets. Similarly, we define a \emph{computable Scott sentence} to be a Scott sentence which is computable.
\end{definition}

All the $\mathcal{L}_{\omega_1,\omega}$ sentences and formulas we mention in this chapter will be computable, so we will say Scott sentence instead of computable Scott sentence.

We say a structure is \emph{computable} if its atomic diagram is computable. We also identify a structure with its atomic diagram. However, the effective Scott theorem is not true, that is, not all computable structures have a computable Scott sentence.

We say an $\mathcal{L}_{\omega_1,\omega}$ formula is $\Sigma_0$ or $\Pi_0$ if it is finitary (i.e.~no infinite disjunction or conjunction) and quantifier free. For $\alpha >0$, a $\Sigma_\alpha$ formula is a countable disjunction of formulas of the form $\exists x \phi$ where $\phi$ is $\Pi_\beta$ for some $\beta < \alpha$. Similarly, a $\Pi_\alpha$ formula is a countable conjunction of formulas of the form $\forall x \phi$ where $\phi$ is $\Sigma_\beta$ for some $\beta < \alpha$. We say a formula is d-$\Sigma_\alpha$ if it is a conjunction of a $\Sigma_\alpha$ formula and a $\Pi_\alpha$ formula. The complexity of Scott sentences of groups will be one of the main topics throughout this chapter.

Another complexity notion we will study is the following:

\begin{definition}
For a structure $\mathcal{A}$, the \emph{index set} $I(\mathcal{A})$ is the set of all indices $e$ such that $\phi_e$ gives the atomic diagram of a structure $\mathcal{B}$ with $\mathcal{B}\cong\mathcal{A}$.
\end{definition}

There is a connection between the two complexity notions that we study:

\begin{proposition}
For a complexity class $\Gamma$, if we have a computable $\Gamma$ Scott sentence for a structure $\mathcal{A}$, then the index set $I(\mathcal{A})$ in $\Gamma$.
\end{proposition}

This proposition and many examples lead to the following thesis:

\begin{quote}
For a given computable structure $\mathcal{A}$, to calculate the precise complexity of $I(\mathcal{A})$, we need a good description of $\mathcal{A}$, and once we have an ``optimal'' description, the complexity of $I(\mathcal{A})$ will match that of the description.
\end{quote}

In this chapter, we focus on the case where the above-mentioned structures are groups. The thesis is shown to be false in \cite{Kn14}, where they found a subgroup of $\mathbb{Q}$ with index set being d-$\Sigma_2$ which cannot have a computable d-$\Sigma_2$ Scott sentence. However,  in the case of finitely-generated groups, the thesis is still open, and the groups we considered give further evidence for the thesis in this case. For more background in computable structure theory, we refer the reader to \cite{As00}.

\subsection{Groups}
We fix the signature of groups to be $\{\cdot, ^{-1}, 1\}$. Throughout the chapter, we will often identify elements (words) in the free group $F_k=F(x_1,\ldots,x_k)$ with functions from $G^k\to G$, by substituting $x_i$ by the corresponding elements from $G$, and do the group multiplication in $G$. Here we restate the relation between word problem and computability of the group:

\begin{theorem}[\cite{Ra60}]
A finitely-generated group is computable if and only if it has solvable word problem.
\end{theorem}

In this chapter, all the groups we consider will be computable.

\subsection{History}

Scott sentences and index sets for many classes of groups have been studied, for example, reduced abelian $p$-groups \cite{Ca06}, free groups \cite{Ca12}, finitely-generated abelian groups, the infinite dihedral group $D_\infty$, and torsion-free abelian groups of rank 1 \cite{Kn}. We will not list all the results, but will mention many of them as needed.

\subsection{Overview of results}

For the reader's convenience, we summarize the main results of each section:
\begin{itemize}
\item (Section 2) Every polycyclic group (including the nilpotent groups) has a computable d-$\Sigma_2$ Scott sentence, and the index set of a finitely-generated non-co-Hopfian nilpotent group is $m$-complete d-$\Sigma_2$.
\item (Section 3) Certain finitely-generated solvable groups, including $(\mathbb{Z}/d\mathbb{Z}) \wr \mathbb{Z}$, $\mathbb{Z} \wr \mathbb{Z}$, and the solvable Baumslag--Solitar groups $BS(1,n)$, have computable d-$\Sigma_2$ Scott sentences and their index sets are $m$-complete d-$\Sigma_2$.
\item (Section 4) The infinitely-generated free nilpotent group has a computable $\Pi_3$ Scott sentence and its index set is $m$-complete $\Pi_3$.
\item (Section 5) We give an example of a subgroup of $\mathbb{Q}$ whose index set is $m$-complete $\Sigma_3$, achieving an upper bound of such groups given in \cite{Kn}.
\item (Section 6) We give another example of a subgroup of $\mathbb{Q}$ which has both computable $\Sigma_3$ and computable $\Pi_3$ pseudo-Scott sentences, but has no computable d-$\Sigma_2$ pseudo-Scott sentence, contrasting a result in the Borel hierarchy of $\operatorname{Mod}(\mathcal{L})$. 
\end{itemize}

\section{Finitely-generated nilpotent and polycyclic groups}

In this section, we will focus on finitely-generated groups, especially nilpotent and polycyclic groups. A priori, even if a structure is computable, it might not have a computable Scott sentence. However, the following theorem says that a computable finitely-generated group always has a computable Scott sentence.

\begin{theorem}[Knight, Saraph, \cite{Kn}]
Every computable finitely-generated group has a computable $\Sigma_3$ Scott sentence.
\end{theorem}

If we think of nilpotent and polycyclic groups as classes of ``tame'' groups, then the abelian groups are the ``tamest'' groups. Using the fundamental theorem of finitely-generated abelian groups, which says that every finitely-generated abelian group is a direct sum of cyclic groups, one can obtain the following theorem saying that the previous computable Scott sentences are not optimal in the case of abelian groups:

\begin{theorem}[Knight, Saraph, \cite{Kn}]
Let $G$ be an infinite finitely-generated abelian group. Then $G$ has a computable d-$\Sigma_2$ Scott sentence. Furthermore, $I(G)$ is $m$-complete d-$\Sigma_2$.
\end{theorem}

This theorem, together with some other results in \cite{Kn}, leads to the question: Does every finitely-generated group have a computable d-$\Sigma_2$ Scott sentence? Generalizing the previous theorem, we show that this is true for polycyclic groups, and also prove completeness for certain classes of groups. We start by giving the definition for several group-theoretic notions that are used in the discussion.

\begin{definition}
For two subgroups $N,M$ of $G$, we write $[N,M]$ to be the subgroup of $G$ generated by all commutators $[n,m]$ with $n\in N$ and $m\in M$.

For a group $G$, we inductively define $G_1 =G$ and $G_{k+1} = [G_k,G]$. We call $G_k$ the $k$-th term in the \emph{lower central series}. A group is called \emph{nilpotent} if $G_{k+1}$ is the trivial group for some $k$, and the smallest such $k$ is called the \emph{nilpotency class} of the group. Note that $G$ is abelian if its nilpotency class equals 1.

We also inductively define $Z_0(G) = 1$ and $Z_{k+1}(G) = \{ x\in G \mid \forall y\in G, [x,y]\in Z_i(G)\}$. We call $Z_k(G)$ the $k$-th term in the \emph{upper central series}. It is well-known that a group $G$ is nilpotent if and only $Z_k(G)=G$ for some $k$. In this case, the smallest such $k$ is equal to to the nilpotency class of the group.

We define the \emph{free nilpotent group} of rank $m$ and class $p$ by $N_{p,m} = F(m) / F(m)_{p+1}$, where $F(m)$ is the free group of $m$ generators.
\end{definition}

\begin{definition}
For a group $G$, we inductively define $G^{(0)} =G$ and $G^{(k+1)} = [G^{(k)},G^{(k)}]$. We call $G^{(k)}$ the $k$-th term in the \emph{derived series}. In the case when $k=1$, this is the \emph{derived subgroup} $G'=G^{(1)}$ of $G$. A group is called \emph{solvable} if $G^{(k)}=1$ for some $k$, and the smallest such $k$ is called the \emph{derived length} of the group. Note that $G$ is abelian if its derived length equals 1.
\end{definition}

\begin{definition}
A polycyclic group is a solvable group in which every subgroup is finitely-generated.
\end{definition}

By definition, every polycyclic group is solvable. It is well known that every finitely-generated nilpotent group is polycyclic. It is also known that all polycylic groups have solvable word problem, and thus are computable. By contrast, an example of a finitely-presented solvable but non-computable group is given in \cite{Ha81}.

We start by giving a sentence saying a tuple generates a subgroup isomorphic to a given finitely-generated computable group $G$. We first fix a presentation $\pres{\overline{a}}{R}$ of $G$, where $R$ is normally closed. The solvability of the word problem of $G$ then says $R$ is computable. Throughout this chapter, we will write $\langle \overline{x} \rangle \cong G$ to be shorthand for $$\bigwedge\limits_{w(\overline{a})\in R} w(\overline{x})=1 \wedge \bigwedge\limits_{w(\overline{a})\notin R} w(\overline{x})\neq1.$$

Since $R$ is computable, the sentence is also computable, and we see it is $\Pi_1$. And since $\overline{x}$ satisfies all the relations of $\overline{a}$ and nothing more, this sentence implies $\langle \overline{x} \rangle \cong \langle \overline{a} \rangle = G$. However, this actually says more -- this sentence requires that these two groups are generated in the same way. Thus, for instance, if $a_1$ is a central element in $G$, then so is $x_1$. This will be a useful observation later. This also implies that the choice of presentation is relevant. In most of our discussion, the choice will be implicit, which is usually the standard presentation (i.e.~the one given in the definition.)

The following is a very useful lemma for finding a computable Scott sentence for finitely-generated groups. We will use this lemma for both polycyclic and solvable groups.

\begin{lemma}[Generating set lemma]\label{generating set lemma}
In a computable group $G$, if there is a non-empty computable $\Sigma_2$ formula $\phi(\overline{x})$ such that every $\overline{x}\in G$ satisfying $\phi$ is a generating tuple of the group, then $G$ has a computable d-$\Sigma_2$ Scott sentence.
\end{lemma}
\begin{proof}
Consider the Scott sentence which is the conjunction of the following:
\begin{enumerate}
\item $\forall\overline{x}\left[\phi(\overline{x}) \rightarrow \forall y \bigvee\limits_{w} w(\overline{x}) = y\right]$
\item $\exists \overline{x} \left[\phi(\overline{x}) \wedge \langle \overline{x} \rangle \cong G\right]$
\end{enumerate}

In (1), $w$ ranges over all words in $\overline{x}$.

Note that (1) is $\Pi_2$ and (2) is $\Sigma_2$, thus the conjunction is d-$\Sigma_2$. To see this is a Scott sentence, pick a group $H$ satisfying the sentence. Then pick a tuple $\overline{x} \in H$ that satisfies the second conjunct. The first conjunct then says $H$ is generated by $\overline{x}$, thus is isomorphic to $G$.
\end{proof}

We now are ready to state and prove our theorem about polycyclic groups, which generalizes the result in \cite{Kn} about infinite finitely-generated abelian groups.

\begin{theorem}\label{polycyclic scott sentence}
Every polycyclic group $G$ has a computable d-$\Sigma_2$ Scott sentence.
\end{theorem}

\begin{proof}

We will show the claim that there is a d-$\Sigma_1$ formula $\phi(\overline{x})$ such that every $\overline{x}\in G$ satisfying $\phi$ is a generating set of the group. We induct on the derived length of $G$.

When the derived length is 1, i.e.~$G$ is abelian, the statement of the theorem was proved in \cite{Kn}, but for the inductive hypothesis, we need to find $\phi$ for $G$. For simplicity, we think of $G$ additively in this case. By the fundamental theorem of abelian groups, suppose $G\cong \mathbb{Z}^n\oplus T$, where $T$ is the torsion part of $G$, and $|T|=k$. Let $\chi(\overline{y})$ be the (finitary) sentence saying the $k$-tuple $\overline{y}$ satisfies the atomic diagram of $T$. Then we consider $\phi(\overline{x},\overline{y})$ to be
$$\chi(\overline{y}) \wedge (\bigwedge\limits_{m>1}m x\neq1) \wedge (\bigwedge\limits_{\det(M)\neq \pm1} \forall \overline{z}\bigwedge\limits_{\langle i_j\rangle\in k^n} M\overline{z}\neq\overline{x}+\langle y_{i_j}\rangle).$$
Here, we use two tuples $\overline{x}$ and $\overline{y}$ for clarity, but one can concatenate them into just one tuple $\overline{x}$.

The first conjunct says that $\overline{y}$ is exactly the $k$ torsion elements in the group. The second conjunct says $\overline{x}$ is torsion-free. In the third conjunct, we are thinking $\overline{z}$, $\overline{x}$, and $\langle y_{i_j}\rangle$ as row vectors, and $M$ ranges over all $n\times n$ matrices with entries in $\mathbb{Z}$ and determinant not equal to $\pm1$. Thus $\bigwedge\limits_{\langle i_j\rangle \in k^n}M\overline{z}\neq\overline{x}+\langle y_{i_j}\rangle$ is really saying $M\overline{z}\neq\overline{x}$ modulo $T$. So, working modulo $T$ and again thinking of the $x_i$'s as row vectors in $\mathbb{Z} \cong G/T$, the third conjunct forces $\overline{x}$, as an $n\times n$ matrix, to have determinant $\pm1$. Thus, $\overline{x}$ is a basis of $G$ modulo $T$, and so every $\overline{x},\overline{y}$ satisfying $\phi$ will generate the group $G$. Finally, we see that the sentence is $\Pi_1$, thus proving the induction base.

Now we prove the induction step. Assume the claim is true for all polycyclic groups with derived length less than that of $G$. In particular, the derived subgroup $G'$ has a computable d-$\Sigma_1$ formula $\phi_{G'}$ as described in the claim. In $G$, $G'$ is defined by the computable $\Sigma_1$ formula
$$G'(x) \equiv \exists \overline{s} \bigvee\limits_{w\in (F_{|\overline{s}|})'}x=w(\overline{s}).$$
Thus, we may relativize $\phi$ by replacing $\exists \overline{x}\theta(\overline{x})$ by $\exists \overline{x} (G'(\overline{x}) \wedge \theta(\overline{x}))$, $\forall \overline{x}\theta(\overline{x})$ by $\forall \overline{x} (G'(\overline{x}) \rightarrow \theta(\overline{x}))$, and adding one more conjunct $\bigwedge\limits_i G'(x_i)$. This does not increase the complexity of the sentences. Furthermore, every element of $G$ satisfying the relativized version $\tilde{\phi}_{G'}$ of $\phi_{G'}$ generates $G'$ in $G$. As in the base case, suppose $G/G' \cong \mathbb{Z}^n\oplus T$, where $T$ is the torsion part, and let $\chi(\overline{y})$ to be the atomic diagram of $T$. We consider $\phi(\overline{x},\overline{y},\overline{z})$ to be the conjunction of the following:
\begin{enumerate}
\item $(\bigwedge\limits_{m>1}m x\ \hat{\neq}\ 1) \wedge (\bigwedge\limits_{\det(M){\neq}  \pm1} \forall \overline{z}\bigwedge\limits_{\langle i_j\rangle \in k^n} M\overline{z}\ \hat{\neq}\ \overline{x}+\langle y_{i_j}\rangle)$
\item $\hat{\chi}(\overline{y})$ 
\item $\tilde{\phi}_{G'}(\overline{z})$
\end{enumerate}

Notice that in (1) we still think of $G/G'$ additively for clarity, while we should really think of it multiplicatively since $G$ is no longer abelian. This is very similar to the sentence in the base case, but everything is relativized. We write $a\ \hat{=}\ b$ to denote that $\exists g (G'(g) \wedge a=b g)$, i.e.~$a$ and $b$ are equal in the quotient group, and this is $\Sigma_1$. And we write $a\ \hat{\neq}\ b$ to denote the negation of $a\ \hat{=}\ b$, which is $\Pi_1$. So, the complexity of (1) is still $\Pi_1$. For $\hat{\chi}(\overline{y})$, again we replace all the $=$ and $\neq$ in $\chi$ by the relativized versions $\hat{=}$ and $\hat{\neq}$, hence making it d-$\Sigma_1$. The relativization doesn't increase the complexity of $\phi_{G'}$, thus the whole conjunct is d-$\Sigma_1$. 

Now (2) says $\overline{y}$ is $T$ in $G/G'$, (1) says $\overline{x}$ generates $\mathbb{Z}^n$ in $G/G'$, and (3) says $\overline{z}$ generates $G'$ in $G$. Thus, $\overline{x}, \overline{y}$, and $\overline{z}$ together generate $G$, hence proving the claim. The theorem now follows from the generating set lemma (Lemma \ref{generating set lemma}).
\end{proof}

We now turn our attention to index sets. We give some results on the completeness of index sets of nilpotent groups, but we need a group-theoretic lemma and a definition first.

\begin{proposition}[Finitely-generated nilpotent group lemma]\label{f.g. nilpotent group lemma}
Every finitely-generated infinite nilpotent group has infinite center. In particular, its center is isomorphic to $\mathbb{Z}\times A$ for some abelian group $A$.
\end{proposition}

\begin{proof}

We induct on the nilpotency class. The statement is obvious when the nilpotency class is 1.

Suppose $N$ is a finitely-generated nilpotent group with finite center. It suffices to show that $N$ is finite. Let the order of the center $Z(N)$ be $k$. Then $Z(N)^k=1$. 

Let the upper central series of $N$ be $1 = Z_0(N) \triangleleft Z_1(N) \triangleleft Z_2(N) \triangleleft \cdots \triangleleft Z_p(N) = N$. For $g\in Z_2(N)$ and $h\in N$, one has $[g,h]\in Z(N)$. Thus, using the identity $[x y,z]=[y,z]^x[x,z]$, we have $[g^k,h]=[g,h]^k=1$, and so $g^k \in Z(N)$, i.e., $Z_2(N)/Z(N)$ has exponent dividing $k$.

Now consider $M=N/Z(N)$. We have $Z(M) = Z_2(N)/Z(N)$, thus has exponent dividing $k$. Since $N$ is finitely-generated and nilpotent, so is $M$, and so is $Z(M)$. Hence $Z(M)$ is finite.

But the nilpotency class of $M$ is less than that of $N$, so by induction hypothesis, $M$ is finite. Then $|N| = |M|\cdot|Z(N)|$ is also finite.
\end{proof}

\begin{definition}
A group is \emph{co-Hopfian} if it does not contain an isomorphic proper subgroup.
\end{definition}

Consider $G$ to be a finitely-generated non-co-Hopfian group. Then let $\phi:G\to G$ to be an injective endomorphism from $G$ onto one of its proper isomorphic subgroups. Then we can form the direct system $\xymatrix{G \ar[r]^\phi &G \ar[r]^\phi &G \ar[r]^\phi &\cdots}$, and write the direct limit as $\hat{G}$. Since every finite subset of $\hat{G}$ is contained in some finite stage, $\hat{G}$ is not finitely-generated, thus is not isomorphic to $G$. This observation will be used later.

We're now ready to prove the completeness result:

\begin{theorem}\label{nilpotent completeness}
The index set of a finitely-generated nilpotent group is $\Pi_2$-hard. Furthermore, the index set of a non-co-Hopfian finitely-generated nilpotent group $N$ is d-$\Sigma_2$-complete.
\end{theorem}

\begin{proof}
We start by proving the second statement. Fix $\phi$ to be an injective endomorphism of $N$ onto one of its proper isomorphic subgroups. Then we apply the construction above to obtain $\hat{N}$.

For a d-$\Sigma_2$ set $S$, we write $S=S_1\smallsetminus S_2$, where $S_1\supseteq S_2$ are both $\Sigma_2$ sets, and we let $S_{1,s}$ and $S_{2,s}$ be uniformly computable sequences of sets such that $n\in S_i$ if and only if for all but finitely many $s$, $n\in S_{i,s}$. Then we construct 
$$G_n\cong
\begin{cases}
\hat{N}, & n\notin S_1\\
N, & n\in S_1\smallsetminus S_2\\
N\times\mathbb{Z}, & n\in S_1\cap S_2.
\end{cases}$$

To build the diagram of $G_n$, at stage $s$, we build a finite part of $G_n$ and a partial isomorphism to one of these three groups based on whether $n\notin S_{1,s}$, $n\in S_{1,s}\smallsetminus S_{2,s}$, or $n\in S_{1,s}\cap S_{2,s}$. It is clear how to build the partial isomorphisms, since all the groups are computable, so we only need to explain how we can change between these groups when $S_{1,s}$ and $S_{2,s}$ change.

To change from $N$ to $\hat{N}$, we apply $\phi$. Note that at every finite stage, the resulting group will still be isomorphic to $N$, but in the limit, it will be $\hat{N}$ if and only if we apply $\phi$ infinitely often, i.e., $n \notin S_1$, as desired.

To change from $N$ to $N\times\mathbb{Z}$, we simply create a new element $a$ that has infinite order and commutes with everything else. To change from $N\times \mathbb{Z}$ to $N$, we choose an element $b$ of infinite order in $Z(N)$ by the finitely-generated nilpotent group lemma (Lemma \ref{f.g. nilpotent group lemma}). We collapse the new element $a$ by equating it with a big enough power of $b$. Again, this will result in the limiting group being $N$ if we collapse $b$ infinitely often, i.e.~$n\notin S_2$, and $N\times\mathbb{Z}$ if we collapse $b$ only finitely often, i.e.~$n\in S_2$.

The second statement follows from doing only the $\Pi_2$ part of the above argument, i.e.~constructing $$G_n\cong
\begin{cases}
N, & n\notin S_2\\
N\times\mathbb{Z}, & n\in S_2.
\end{cases}$$
\end{proof}

Note that here we do not have a completeness result for the class of co-Hopfian finitely-generated nilpotent groups. For a discussion about this ad-hoc class of groups, we refer the readers to \cite{Be03}. However, ``most'' finitely-generated nilpotent groups are non-co-Hopfian, including the finitely-generated free nilpotent groups, and we have the following:

\begin{corollary}
The index set of a finitely-generated free nilpotent group is d-$\Sigma_2$-complete.
\end{corollary}

\begin{note}
Using the nilpotent residual property (Lemma \ref{nilpotent residual property}), we can show the $\text{d-}\Sigma_2$ completeness result for free nilpotent groups within the class of free nilpotent groups, provided that the number of generators is more than the step of the group. For the definition and more discussion on the complexity within a class of groups, we refer the reader to \cite{Ca12}.
\end{note}

To close this section, we state a proposition about co-Hopfian and non-co-Hopfian groups.

\begin{proposition}
The index set of a computable finitely-generated non-co-Hopfian group is $\Sigma_2$-hard. On the other hand, a computable finitely-generated co-Hopfian group $G$ has a computable d-$\Sigma_2$ Scott sentence.
\end{proposition}

\begin{proof}
The first statement is proved by running the $\Sigma_2$ part of the argument in Theorem \ref{nilpotent completeness}.

For the second statement, consider the computable $\Pi_1$ formula $\phi(\overline{x}) \equiv \langle \overline{x} \rangle \cong G$. This is a non-empty formula, and since $G$ is co-Hopfian, every realization of $\phi$ in $G$ generates $G$. Thus by the generating set lemma (Lemma \ref{generating set lemma}), $G$ has a computable d-$\Sigma_2$ Scott sentence.
\end{proof}

\section{Some examples of finitely-generated solvable groups}

In this section, we continue to look at the bigger, but still somewhat tame, class of finitely-generated solvable groups. Note that even though the class of solvable groups is closed under subgroups, the class of finitely-generated solvable groups is not. This leads to an inherent difficulty when dealing with solvable groups, namely a group could possibly contain a higher-complexity subgroup. For example, the lamplighter group, which we shall define later and prove to have a computable d-$\Sigma_2$ Scott sentence, contains a subgroup isomorphic to $\mathbb{Z}^\omega$, whose index set is $m$-complete $\Pi_3$.

We start this section with the definition of the (regular, restricted) wreath product, which is a technique often used in group theory to construct counterexamples:

\begin{definition}
For two groups $G$ and $H$, we define the \emph{wreath product} $G\wr H$ of $G$ by $H$ to be the semidirect product $B\rtimes H$, where the \emph{base group} $B$ is the direct sum of $|H|$ copies of $G$ indexed by $H$, and the action of $H$ on $B$ is by shifting the coordinates by left multiplication.
\end{definition}

One important example of a finitely-generated solvable group is the lamplighter group. It is usually defined as the wreath product $\mathbb{Z}/2\mathbb{Z} \wr \mathbb{Z}$, but we will be looking at two generalizations of it, $(\mathbb{Z}/d\mathbb{Z}) \wr \mathbb{Z}$ and $\mathbb{Z} \wr \mathbb{Z}$.

\begin{theorem}
The \emph{lamplighter groups} $L_d = (\mathbb{Z}/d\mathbb{Z}) \wr \mathbb{Z}$ each have computable d-$\Sigma_2$ Scott sentences. Furthermore, their index sets are $m$-complete d-$\Sigma_2$.
\end{theorem}

\begin{proof}
To find the Scott sentence, we will use the generating set lemma (Lemma \ref{generating set lemma}). Consider the formula $$\phi(a,t) \equiv (\langle a,t\rangle\cong L_d) \wedge (\forall s \bigwedge\limits_{i>1} a^t \neq a^{(s^i)}) \wedge (\forall b \bigwedge\limits_{\overline{k}}\prod\limits_{i}(b^{k_i})^{t^{i}}\neq a).$$

In the second conjunct, $s$ ranges over the group elements. In the second infinite conjunction, $\overline{k}$ ranges over all sequences in $\mathbb{Z}$ indexed by $\mathbb{Z}$ and has only finitely many, but at least two, nonzero entries. We first observe that the standard generator satisfies this formula, so $\phi$ does not define the empty set. 

Now let $(a,t)$ be a tuple satisfying $\phi$. The first conjunct implies that $a$ is in the base group, and $t$ is not in the base group. The second conjunct says that if we think of $t$ as an element of the semidirect product $L_d\cong (\mathbb{Z}/d\mathbb{Z})^\mathbb{Z} \rtimes \mathbb{Z}$, then the $\mathbb{Z}$-coordinate of $t$ is $\pm 1$. The third coordinate then says that $a$ does not have more than one nonzero entries, and hence (by the first conjunct) the only nonzero entry must be co-prime to $d$. Thus, $a$ generates a copy of $\mathbb{Z}/d\mathbb{Z}$. Using conjugation by $t$ to generate the other copies of $\mathbb{Z}/d\mathbb{Z}$, we see $a$ together with $t$ generate the whole group. So by the generating set lemma (Lemma \ref{generating set lemma}), $L_d$ has a computable d-$\Sigma_2$ Scott sentence.

To show completeness of the index set, fix $\phi$ to be an injective endomorphism of $L_d$ onto one of its proper isomorphic subgroups, say mapping the standard generators $(a,t)$ to $(a,t^2)$. Let $\hat{L}_d$ be the direct limit of $\xymatrix{L_d \ar[r]^\phi &L_d \ar[r]^\phi &L_d \ar[r]^\phi &\cdots}$.

For a d-$\Sigma_2$ set $S$, we write $S=S_1\smallsetminus S_2$, where $S_1\supseteq S_2$ are both $\Sigma_2$ sets, and we let $S_{1,s}$ and $S_{2,s}$ be uniformly computable sequences of sets such that $n\in S_i$ if and only if for all but finitely many $s$, $n\in S_{i,s}$. Then we construct 
$$G_n\cong
\begin{cases}
\hat{L}_d, & n\notin S_1\\
L_d, & n\in S_1\smallsetminus S_2\\
(\mathbb{Z}/d\mathbb{Z})\wr\mathbb{Z}^2, & n\in S_1\cap S_2.
\end{cases}$$

As in the nilpotent case (Theorem \ref{nilpotent completeness}), we build a partial isomorphism of one of these groups in stages. To change between $L_d$ and $\hat{L}_d$ is the same as before, we apply $\phi$ whenever $n\notin S_{1,m}$, and keep building $L_d$ otherwise.

To change from $L_d$ to $(\mathbb{Z}/d\mathbb{Z})\wr\mathbb{Z}^2$, we create a new element $s$ to be the other generator of $\mathbb{Z}^2$, and equate it with a big enough power of $t$ to change back. This will result in the limiting group being $L_d$ if we collapse $s$ infinitely often, i.e.~$n\notin S_2$, and $(\mathbb{Z}/d\mathbb{Z})\wr\mathbb{Z}^2$ otherwise.
\end{proof}

\begin{theorem}
Let $L = \mathbb{Z} \wr \mathbb{Z}$. Then $L$ has a computable d-$\Sigma_2$ Scott sentence, and $I(L)$ is $m$-complete d-$\Sigma_2$.
\end{theorem}

\begin{proof}[Sketch of proof]
The proof is essentially the same as above. The $\Pi_1$ formula will be $$\phi(a,t) \equiv (\langle a,t\rangle\cong L) \wedge (\forall s \bigwedge\limits_{i>1} a^t \neq a^{(s^i)}) \wedge (\forall b \bigwedge\limits_{\overline{l},\overline{k}}\prod\limits_{i}(b^{k_i})^{t^{l_i}}\neq a).$$
The only difference is that we will also allow $\overline{k}$ to have only one nonzero entry which is not $\pm1$, in addition to $\overline{k}$'s which have at least two nonzero entries. This is to rule out the case where $a$ is a power of the standard generator.

For completeness, we will construct 
$$G_n\cong
\begin{cases}
\hat{L}, & n\notin S_1\\
L, & n\in S_1\smallsetminus S_2\\
\mathbb{Z}\wr\mathbb{Z}^2, & n\in S_1\cap S_2.
\end{cases}$$
\end{proof}

\begin{remark}
In fact, this theorem can be generalized to $\mathbb{Z}^n \wr \mathbb{Z}^m$. However, we will omit the proof in the interest of space. We have to add into $\phi(\overline{a},\overline{t})$ extra conjuncts to make sure that $\overline{a}$ generates a copy of $\mathbb{Z}^n$ and $\overline{t}$ generates $\mathbb{Z}^m$ modulo the base group, and the extra conjuncts are similar to what we did in the polycyclic case (Theorem \ref{polycyclic scott sentence}). In proving completeness, we use the direct limit and $\mathbb{Z}^n \wr \mathbb{Z}^{(m+1)}$ as the alternate structures.
\end{remark}

We now look at another class of groups, the Baumslag--Solitar groups, which are very closely related to the lamplighter groups. Indeed, in \cite{St06}, it was shown that $BS(1,n)$ converges to $\mathbb{Z} \wr \mathbb{Z}$ as $n\to\infty$. We shall see great similarity in the arguments used for these groups also.

\begin{definition}
The \emph{Baumslag--Solitar groups $BS(m,n)$} are two-generator one-relator groups given by the presentation:
$$BS(m,n)=\pres{a,b}{b a^m b^{-1} = a^n}$$
Note that $BS(m,n) \cong BS(n,m)$.
\end{definition}

\begin{theorem}
$BS(m,n)$ is solvable if and only if $|m|=1$ or $|n|=1$, in which case it is also not polycyclic and its derived length is 2.
\end{theorem}

\begin{theorem}
For each $n$, the solvable Baumslag--Solitar group $BS(1,n)$ has a computable d-$\Sigma_2$ Scott sentence. Furthermore, its index set is $m$-complete d-$\Sigma_2$ for every n.
\end{theorem}

\begin{proof}
$BS(1,n)$ has the semidirect product structure $B \rtimes \mathbb{Z}$ where $B=\mathbb{Z}[\frac{1}{n},\frac{1}{n^2},\dots]=\{ \displaystyle\frac{x}{y} : y|n^k \text{ for some }k\}$, and the action of $1\in\mathbb{Z}$ on $B$ is by multiplication by $n$. Again we are abusing notation by writing $BS(1,n)$ multiplicatively but $B$ additively.

For finding the Scott sentence, we consider the formula $$\phi(a,t)\equiv (\langle a,t\rangle\cong BS(1,n)) \wedge (\forall b \bigwedge\limits_{\gcd(i,n)=1} b^i \neq a).$$
This is not empty because the standard generators $(1,0)$, $(0,1)\in B\rtimes\mathbb{Z}$ satisfy it.

For a tuple $(a,t)$ satisfying $\phi$, the first conjunct guarantees that $a$ is in the base group and $t$, as an element of the original Baumslag--Solitar group $BS(1,n)=B\rtimes\mathbb{Z}$, has the $\mathbb{Z}$ coordinate being $1$ in the semidirect product, because of the inclusion of the formula $tat^{-1} = a^d$. The second conjunct guarantees that the $B$ coordinate of $a$ is $\frac{x}{y}$ for some $x$ and $y$ both dividing some power of $n$. Thus appropriately conjugating $a$ by $t$, we see $\frac{1}{y'} \in \langle a,t\rangle$ for some $y'$, thus $1\in \langle a,t \rangle$, and hence $a$ and $t$ generate the whole group. By the generating set lemma (Lemma \ref{generating set lemma}), we obtain a computable d-$\Sigma_2$ Scott sentence for $BS(1,n)$.

To show completeness, we first observe $BS(1,n)$ is not co-Hopfian. Indeed, let $p$ be a prime not dividing $n$, then consider the endomorphism sending $a$ to $a^p$ and fixing $b$. This is injective but not surjective because, for instance, it misses the element $1$ in $B$. Thus, we construct 
$$G_n\cong
\begin{cases}
\widehat{BS(1,n)}, & n\notin S_1\\
BS(1,n), & n\in S_1\smallsetminus S_2\\
(B^\mathbb{Z})\rtimes(\mathbb{Z}^2), & n\in S_1\cap S_2,
\end{cases}$$
where in $(B^\mathbb{Z})\rtimes(\mathbb{Z}^2)$, $B^\mathbb{Z}$ is the direct sum of countably many copies of $B$, indexed by $\mathbb{Z}$, and the action of the first coordinate of $\mathbb{Z}^2$ is by multiplying by $n$ (to each coordinate), and the action of the second coordinate is by shifting the copies of $B$. The same argument as above will show this construction gives d-$\Sigma_2$ completeness of $I(BS(1,n))$.
\end{proof}

\section{Infinitely-generated free nilpotent groups}

We will now turn our attention to infinitely-generated groups. In this section, we start with a natural continuation of Section 2, showing that the infinitely-generated free nilpotent groups $N_{p,\infty}$ have a computable $\Pi_3$ Scott sentence, and their index sets are $m$-complete $\Pi_3$. We start by stating the following result for $p=1$:

\begin{theorem}[\cite{Ca06}]
The infinitely-generated free abelian group $\mathbb{Z}^\omega$ has a computable $\Pi_3$ Scott sentence. Furthermore, $I(\mathbb{Z}^\omega)$ is $m$-complete $\Pi_3$.
\end{theorem}

To find a computable Scott sentence for the infinitely-generated free nilpotent group, we give a lemma analogous to the generating set lemma (Lemma \ref{generating set lemma}).

\begin{lemma}[Infinite generating set lemma]\label{infinite generating set lemma}
Suppose $G\cong\pres{a_1, a_2, \dots}{R}$, where $R$ is a normal subgroup of $F_\omega$. Let $R_i$ be $R\cap F_{a_1,\ldots,a_i}\subset F_\omega$. If there are $\langle\gamma_k\rangle_{k\in\omega}$ such that
\begin{enumerate}
\item $\gamma_k(\overline{x})$ implies $\langle \overline{x} \rangle \cong \pres{a_1,a_2,\dots,a_k}{R_k}$ modulo the theory of groups.
\item $G \models \exists x_1\ \gamma_1(x_1)$.
\item $\displaystyle G \models \bigwedge\limits_{k}(\forall x_1,\dots,x_k[\gamma_k(x_1,\dots,x_k)\to\forall y  \bigvee\limits_{l\geq k+1} \exists x_{k+1},\dots,x_l \ \gamma_l(x_1,\dots,x_l)\wedge z\in \langle x_1,\dots,x_l \rangle])$ (``Every $\gamma_k(\overline{x})$ can be `extended' ''. In a countable group, this implies $\overline{x}$ can be extended to a basis.)
\end{enumerate}
Then $\phi$, the conjunction of the group axioms and the sentences (2) and (3), is a Scott sentence of $G$.
\end{lemma}
\begin{proof}
By assumption, $G$ models $\phi$. Let $H$ be a countable group modeling $\phi$. We first choose $x_1$ by (2). Note that (3) allows us to extend any $\overline{x}$ satisfying $\gamma_k$ to generate any element in the group $H$. Thus, we enumerate $H$, and iteratively extend $\overline{x}$ to generate the whole group $H$. If we consider the relations that hold on the infinite limiting sequence $\overline{x}$, the relation on $x_1,\ldots,x_k$ is exactly $R_k$ by (1), thus the group $H = \langle \overline{x} \rangle \cong \pres{\overline{a}}{R} = G$.
\end{proof}

\begin{corollary}
$N_{p,\infty}$ has a computable $\Pi_3$ Scott sentence.
\end{corollary}
\begin{proof}
Let $\gamma_k(\overline{y}) \equiv (\langle \overline{y} \rangle \cong N_{p,k})  \wedge (\forall \overline{z} \bigwedge\limits_{\det(M) \neq \pm 1} M\overline{z} \ \hat\neq\  \overline{y})$, where in the second conjunct the inequality is relativized (as in the polycyclic case, Theorem \ref{polycyclic scott sentence}) to the abelianization $H/H'$ and we are abusing notation and thinking of the abelianization additively. So for $M\overline{z}$ we mean matrix multiplication, thinking of each $z_i$ as a row vector.

To show the $\gamma_k$'s satisfy the extendibility condition, fix $\overline{x}$ satisfying $\gamma_k$. Working in the abelianization, after truncating the columns in which $\overline{x}$ has no nonzero entries	, we write $\overline{x}$ in its Smith normal form, i.e.~find invertible $k\times k$ and $n \times n$ matrices $S$ and $T$ such that $S\overline{x}T$ has all but the $(i,i)$-th entries being zero. The second conjunct of $\gamma_k$ guarantees that all the $(i,i)$-th entries are actually 1. Thus, $\overline{x}$ can be extended to a basis of $\mathbb{Z}^\infty = \operatorname{ab}(N_{p,\infty})$. By a theorem of Magnus (\cite[Lemma 5.9]{Ma04}), this implies that $\overline{x}$ can be extended to a basis of $N_{p,\infty}$, thus satisfying the extendibility condition.

The $\gamma_k$'s are $\Pi_1$, and a direct counting shows that the Scott sentence we obtain from the previous lemma is $\Pi_3$.
\end{proof}

For completeness of the index set of $N_{p,\infty}$, we generalize the technique from the abelian case, but we will need the following group-theoretic lemma.

\begin{lemma}[Nilpotent residual property]\label{nilpotent residual property}
For $n,m\geq p$, $N_{p,n}$ is fully residually-$N_{p,m}$. I.e., for every finite subset $S\subset N_{p,n}$, there exists a homomorphism $\phi: N_{p,n}\to N_{p,m}$ such that $\phi$ is injective on $S$.
\end{lemma}

\begin{proof}
Baumslag, Myasnikov, and Remeslennikov in \cite{Ba99} showed that any group universally equivalent to (i.e., having the same first-order universal theory as) a free nilpotent group $N_{p,m}$ is fully residually-$N_{p,m}$. Timoshenko in \cite{Ti00} showed that $N_{p,n}$ and $N_{p,m}$ are universally equivalent for $n,m \geq p$. Combining these two results, we prove the desired lemma.

This is a combination of the result of Baumslag, Myasnikov, and Remeslennikov \cite{Ba99} that any group universally equivalent to (i.e., having the same first-order universal theory as) 
a free nilpotent group $N_{p,m}$ is fully residually-$N_{p,m}$, and the result of Timoshenko \cite{Ti00} that $N_{p,n}$ and $N_{p,m}$ are universally equivalent for $n,m \geq p$.
\end{proof}

\begin{corollary}
$I(N_{p,\infty})$ is $m$-complete $\Pi_3$ for every $p$.
\end{corollary}

\begin{proof}
Recall that $\operatorname{COF}$, the index set of all cofinite c.e.~sets, is m-complete $\Sigma_3$. We will reduce the complement of $\operatorname{COF}$ to $I(N_{p,\infty})$.

We construct $G_n$ uniformly in $n$. We first fix an infinite set of generators $a_0, a_1, \dots$ and $p$ distinguished generators $b_0,b_1,\dots, b_p$ distinct from the $a_i$'s, and we start the construction by constructing $\overline{b}$. At a finite stage, if we see some natural number $k$ being enumerated into $W_n$, we collapse $a_k$ by taking the subgroup $N_{p,m}\subset N_{p,\infty}$ generated by all the generators that have been mentioned so far. Having the $\overline{b}$ means $m> p$, so by the nilpotent residual property (Lemma \ref{nilpotent residual property}), we can embed what we have constructed so far into $N_{p,m-1}$ with the same generators except $a_k$. 

Thus, in the limit, if $n\in \COF$, then we will collapse all but finitely many $a_i$'s, hence the limiting group will be a finitely-generated (free) nilpotent group not isomorphic to $N_{p,\infty}$; and if $n\notin \COF$, we will still have infinitely many of the $a_i$'s, and the limiting group will be isomorphic to $N_{p,\infty}$. This shows $I(N_{p,\infty})$ is $m$-complete $\Pi_3$.
\end{proof}

\begin{remark}
In the corollary, one actually can prove that $I(N_{p,\infty})$ is $m$-complete $\Pi_3$ within the class of free nilpotent groups. The interested reader can compare this to the same result for infinitely-generated free abelian groups in \cite{Ca06}.
\end{remark}

\section{A subgroup of $\mathbb{Q}$}

In this section, we will look at a special subgroup of $\mathbb{Q}$. Knight and Saraph \cite[\S 3]{Kn} considered subgroups of $\mathbb{Q}$, distinguishing between cases by looking at the following invariants.

\begin{definition}
We write $P$ to be the set of primes. Let $G$ be a computable subgroup of $\mathbb{Q}$. Without loss of generality, we will assume $1\in G$, otherwise we can take a subgroup of $\mathbb{Q}$ isomorphic to $G$ containing $1$. We define:
\begin{enumerate}
\item $\displaystyle P^0(G) = \{ p\in P : G \models p \nmid 1 \}$
\item $\displaystyle P^{\text{fin}}(G) = \{ p\in P : G\models p|1\text{ and }p^k\nmid 1\text{ for some }k \}$
\item $\displaystyle P^{\infty}(G)=\{ p\in P: G\models p^k|1 \text{ for all }k\}$
\end{enumerate}
\end{definition}

\begin{remark}
\begin{enumerate}
\item Define $\displaystyle P^{k}(G) = \{ p\in P : G\models p^k|1 \text{ and } p^{k+1}\nmid 1\}$. Then two subgroups $G$, $H$ of $\mathbb{Q}$ are isomorphic if and only if $P^k(G)=^*P^k(H)$ for every $k$, with equalities holding on cofinitely many of $k$, and $P^\infty(G)=P^\infty(H)$. $S=^*T$ means $S$ and $T$ only differ by finitely many elements.
\item Since $G$ is computable, $P^0$ is $\Pi_1$, $P^\text{fin} \cup P^\infty$ is $\Sigma_1$, $P^\text{fin}$ is $\Sigma_2$, and $P^\infty$ is $\Pi_2$.
\end{enumerate}
\end{remark}

Dividing the subgroups of $\mathbb{Q}$ into cases by these invariants, Knight and Saraph determined the upper and lower bound of complexities of Scott sentences and the index sets in some cases. The case we consider here is when $P^0$ is infinite, $P^\text{fin}$ is finite (and thus, without loss of generality, empty), and $P^\infty$ is infinite. This is case 5 in \cite{Kn}, and they have the following results:

\begin{theorem}[\cite{Kn}]
Let $G$ be a computable subgroup of $\mathbb{Q}$ with $|P^0|=\infty$, $P^\text{fin}=\emptyset$, and $|P^\infty|=\infty$. Then
\begin{enumerate}
\item $G$ has a computable $\Sigma_3$ Scott sentence.
\item $I(G)$ is d-$\Sigma_2$-hard.
\item If $P^\infty$ is low, then $I(G)$ is d-$\Sigma_2$.
\item If $P^\infty$ is not high$_2$, then $I(G)$ is not $m$-complete $\Sigma_3$.
\end{enumerate}
\end{theorem}

It was not known that if there is a subgroup of $\mathbb{Q}$ as in the theorem that has $m$-complete $\Sigma_3$ index set, thus achieving the upper bound in (1). In Proposition \ref{halting set}, we shall give such an example.

Also, the following theorem shows that such a group does not have a computable d-$\Sigma_2$ Scott sentence unless $P^0$ is computable.

\begin{theorem}[\cite{Kn14}]
Let $G$ be a computable subgroup of $\mathbb{Q}$ with $|P^0|=\infty$, $P^\text{fin}=\emptyset$, and $|P^\infty|=\infty$, and suppose $P^\infty$ is not computable. Then $G$ does not have a computable d-$\Sigma_2$ Scott sentence.
\end{theorem}

Thus, when $P^\infty$ is low but not computable, this gives a negative answer to the conjecture that the complexity of the index set should equal to the complexity an optimal Scott sentence. Continuing in this direction, we give an example of a subgroup in this case where the two complexities do equal each other, and are both $\Sigma_3$.

\begin{proposition}\label{halting set}
Let $K$ be the halting set. Let $G \subseteq \mathbb{Q}$ be a subgroup such that $1 \in G$, $P^\infty(G)=\{ p_n\in P \mid n\in K\}$, and $P^\text{fin}(G)=\emptyset$. Then $I(G)$ is $m$-complete $\Sigma_3$. 
\end{proposition}

\begin{proof}
Fix $n$. We construct $G_n$ so that $G_n\cong G$ if and only if $n\in\COF$.

For every $s$, we can recursively find the index $k_s = e$ of a program such that
$$\phi_e(e) =
\begin{cases}
\downarrow, & \text{if }\phi_n(s)\downarrow \\
\uparrow, & \text{if }\phi_n(s)\uparrow .
\end{cases}
$$

Now we construct $G_n$ by making $p_{k_s}|1$ for every $k_s$. We also make $p_i$ divide every element if we see $i\in K$.

Now we verify $G_n \cong G$ if and only if $n\in \COF$. We first observe that $P^\infty(G_n) = \{p_i \mid i\in K\}$. It's also clear from construction that $P^\text{fin}(G_n) = \{ p_{k_s} \mid k_s \notin K\}$. But $k_s \in K$ if and only if $\phi_n(s)\downarrow$, thus $P^\text{fin}(G) = \{ p_{k_s} \mid \phi_n(s)\uparrow \}$.

Now, $P^\infty(G_n) = \{p_i \mid i\in K\} = P^\infty(G)$. Thus $G_n \cong G$ iff $P^\text{fin}(G_n) =^* P^\text{fin}(G) = \emptyset$ iff $P^\text{fin}(G_n) = \{ p_{k_s} \mid \phi_n(s)\uparrow \}$ is finite iff $n \in \COF$.
\end{proof}

\begin{remark}
Note that this argument works for any $X \equiv_m K$. It is natural to then ask that whether we can find a Turing-degree based characterization of when the index set will be $m$-complete $\Sigma_3$. In the next section, we will show this cannot be found.
\end{remark}

\section{Complexity hierarchy of pseudo-Scott sentences}

In this section, we continue looking at subgroups of $\mathbb{Q}$ as above. We first give the definition of a pseudo-Scott sentence, which, just like a Scott sentence, identifies a structure, but only among the computable structures. Note that every computable Scott sentence is a pseudo-Scott sentence.

\begin{definition}
A \emph{pseudo-Scott sentence} for a structure $\mathcal{A}$ is a sentence in $\mathcal{L}_{\omega_1,\omega}$ whose computable models are exactly the computable isomorphic copies of $\mathcal{A}$.
\end{definition}

Similar to the case of computable Scott sentences, a pseudo-Scott sentence of a structure yields a bound on the complexity of the index set of the structure.

We shall give an example of a group which has a computable $\Sigma_3$ pseudo-Scott sentence and a computable $\Pi_3$ pseudo-Scott sentence, but no computable d-$\Sigma_2$ pseudo-Scott sentence. This is related to a question about the effective Borel hierarchy in $\operatorname{Mod}(\mathcal{L})$.

Consider the complexity hierarchy of (computable pseudo-)Scott sentences. Since $\alpha \smallsetminus (\beta \smallsetminus \gamma) = (\alpha\wedge\neg\beta) \vee (\alpha \wedge \gamma)$, we see that the hierarchy collapses in the sense that the complexity classes $k$-$\Sigma_n$ are all the same for $k\geq2$.

One open question is whether the complexity classes $\Delta_{n+1}$ and d-$\Sigma_n$ are the same or not for regular, computable, and pseudo-Scott sentences. The complexity hierarchy of Scott sentences (computable Scott sentences, respectively) is related to the boldface (effective, respectively) Borel hierarchy on the space $\operatorname{Mod}(\mathcal{L})$, see \cite{Va75} and \cite{Va07}. In \cite{Mi78}, it was shown that $\boldsymbol{\Delta}_{n+1}$ and d-$\boldsymbol{\Sigma}_n$ are the same in the boldface case, i.e.~if a structure has a $\Sigma_{n+1}$ Scott sentence and a $\Pi_{n+1}$ Scott sentence, then it also has a d-$\Sigma_n$ Scott sentence. This gives a positive answer to the question for Scott sentences. However, we will prove that this is not true in the complexity hierarchy of computable pseudo-Scott sentences by giving a subgroup $G\subset\mathbb{Q}$ which has a computable $\Sigma_3$ pseudo-Scott sentence and a computable $\Pi_3$ pseudo-Scott sentence, but no computable d-$\Sigma_2$ pseudo-Scott sentence. This gives a negative answer to the question for pseudo-Scott sentences. The question of whether $\Delta_{n+1}$ and d-$\Sigma_n$ are the same in the effective case (the complexity hierarchy of computable Scott sentence) remains open.

We start by strengthening a result in \cite{Kn14}. The first part of the proof where we construct the theory $T$ is unchanged.

\begin{lemma}\label{lower bound lemma}
Fix a non-computable c.e.~set $X$, and let $G\subset \mathbb{Q}$ be a subgroup such that $1\in G$, $P^\infty(G) = \{p_i\mid i\in X\}$, and $P^\text{fin}(G)=\emptyset$. Then $G$ does not have a computable d-$\Sigma_2$ pseudo-Scott sentence.
\end{lemma}

\begin{proof}
Suppose $G$ has a computable d-$\Sigma_2$ pseudo-Scott sentence $\phi \wedge \psi$, where $\phi$ is computable $\Pi_2$ and $\psi$ is computable $\Sigma_2$. Let $\alpha$ be a computable $\Pi_2$ sentence characterizing the torsion-free abelian groups $A$ of rank 1 such that $P^\infty(A)\subseteq X$. By \cite[Lemma 2.3]{Kn14}, $\alpha\vdash\phi$, thus we can replace $\phi$ by $\alpha$ in the pseudo-Scott sentence.

Also, again by \cite{Kn14}, we can assume $\psi$ has the form $\exists x\  \chi(x)$ where $x$ is a singleton and $\chi(x)$ is a c.e.~conjunction of finitary $\Pi_1$ formulas.

Now we consider the first-order theory $T$ in $\mathcal{L}$ with an extra constant symbol $c$ to be union of the following sentences:
\begin{enumerate}
\item axioms of torsion-free abelian groups,
\item $\forall x \exists y \ p y=x$ for each $p\in X$,
\item $\rho_i(c)$ for every finitary $\Pi_1$ conjunct $\rho_i(x)$ of the $\Pi_1$ sentence $\chi(x)$.
\end{enumerate}

Now we show that \cite[Lemma 2.4]{Kn14} is still true:

\begin{claim}For every $i\notin X$, there is some $k$ such that $T\vdash p^k_i \nmid c$. \end{claim}

\begin{proof}[Proof of claim]
Suppose this is not true. There is $n\notin X$ such that $T\cup\{p^k_n|c\}_{k=1}^\infty$ is consistent. Take a model $H\models T\cup\{p^k_n | c\}_{k=1}^\infty$, and let $C\subset H$ be the subgroup consisting of rational multiples of $c$. 

Let $K\subset\mathbb{Q}$ be the computable group with $1\in K$, $P^\infty(K) = \{p_i\mid i\in X\}\cup\{p_n\}$, and $P^\text{fin}(K)=\emptyset$. Then $K$ is isomorphic to a subgroup of $C$. By interpreting the constant symbol $c$ in $K$ as the preimage of $c\in C$, $K$ is a substructure of $H$. Thus all the finitary $\Pi_1$ statements $\rho_i(c)$ are also true in $K$. 

So, $K$ is a torsion-free rank 1 abelian group satisfying $T$, thus $K\models\phi\wedge\psi$. Since $K$ is computable, it is isomorphic to $G$, but $P^\infty(G) \neq P^\infty(K)$, a contradiction.
\end{proof}

Now we have a computable theory $T$ such that for every $i\notin X$, there is some $k$ such that $T\vdash p^k_i \nmid c$. Therefore, the complement of $X$ is c.e., and this contradicts the assumption that $X$ is non-computable c.e.
\end{proof}

We also need the following lemma:

\begin{lemma}
There exists a c.e.~set $X \subseteq \omega$ with $X \equiv_T 0'$ satisfying the following:\\
There is a uniformly c.e.~sequence $S_n$ so that if $W_n\supset X$, $W_n\neq^*X$, then $S_n$ is an infinite c.e.~subset of $W_n\smallsetminus X$.
\end{lemma}

\begin{proof}
Write $I_i =  [\frac{i(i+1)}{2},\frac{(i+1)(i+2)}{2})$. Note that $|I_i|=i+1$. Consider the following requirements:

\begin{itemize}
\item $R_i$: $X\cap I_i \neq \emptyset$ if and only if $i\in 0'$
\item $Q_k$: build an infinite c.e.~$S_k\subset W_k\smallsetminus X$ if $W_k\supset X$ and $W_k\neq^* X$
\end{itemize}

If at some stage $R_i$ sees $i\in 0'$, then it puts some element of $I_i$ that is not yet blocked by higher priority $Q_k$'s into $X$.

$Q_k$ will attempt to put $n$ into $S_k$ whenever $n \in W_k\smallsetminus X$ at stage $s$. Suppose $n\in I_i$. If $i<k$, then $Q_k$ does nothing. If $i>k$, then $Q_k$ puts $n$ into $S_k$, and blocks $R_i$ from enumerating $n$ into $X$, but $Q_k$ will also block itself from enumerating other elements of $I_i$ into $S_k$.

Note that for each $R_i$, at most $i$ elements of $I_i$ will be blocked, because for each $I_i$, every higher priority $Q_k$ will block at most one element. Thus $R_i$ can always satisfy the requirement.

Also, if $W_k\supset X$ and $W_k\neq^* X$, then $Q_k$ will eventually enumerate infinitely many numbers into $S_k$, since after enumerating finitely many of them, there are only finitely many things blocked by $Q_k$ itself and finitely many higher priority $R_i$'s. Lastly, $S_k$ will be disjoint from $X$ because whenever $n$ is enumerated into $S_k$, the $R_i$'s will be blocked from enumerating it into $X$.
\end{proof}

\begin{theorem}\label{counterexample}
There exists a group with both computable $\Sigma_3$ and computable $\Pi_3$ pseudo-Scott sentences (i.e.~$\Delta_3$), but no computable d-$\Sigma_2$ pseudo-Scott sentence.
\end{theorem}

\begin{proof}
Choose $X$ as in the previous lemma. Consider the subgroup $G\subset\mathbb{Q}$ with $P^\text{fin}(G) = \emptyset$ and $P^\infty(G)=X$. By \cite{Kn14}, $G$ has a computable $\Sigma_3$ (pseudo-)Scott sentence. By Lemma \ref{lower bound lemma}, $G$ does not have a computable d-$\Sigma_2$ pseudo-Scott sentence.

Let $\phi$ be the conjunction of $\bigwedge\limits_{S_k} \exists g \bigwedge\limits_{p\in S_k}p \nmid g$, ``$G$ is a subgroup of $Q$'', and ``$P^\infty(G) \supseteq X$''.

Say $H\models \phi$ be a computable group. Then $P^\infty(H)\cup P^\text{fin}(H)$ is c.e., so let $n$ be such that $W_n = P^\infty(H)\cup P^\text{fin}(H)$. Note that $W_n\supset X$. If $H \ncong G$, then $W_n\neq^* X$. Now consider $S_n \subset Y\smallsetminus X$, and the corresponding conjunct in $\phi$ which says $\exists g \bigwedge\limits_{p\in S_Y}p \nmid g$. But every element in $H$ is divisible by all but finitely many elements from $W_n$, and $S_n$ is an infinite subset of $W_n$, so $H$ cannot model this existential sentence, a contradiction. Thus $\phi$ is a computable $\Pi_3$ pseudo-Scott sentence of $G_X$.
\end{proof}

\begin{remark}
Note that in the first countable conjunction, the set of indices of $S_Y$ for $Y\supset X$ and $Y \neq^* X$ is not c.e. However, the set of indices of $S_Y$ for all $Y\subset \omega$ is c.e., even computable, by construction, and $G_X$ still models $\phi$ for this bigger conjunction.
\end{remark}

Note that the two groups in Proposition \ref{halting set} and Theorem \ref{counterexample} both have $P^{\text{fin}}=\emptyset$ and $P^\infty \equiv_T 0'$. However, one of them has index set being $m$-complete $\Sigma_3$, while the other has index set being $\Delta_3$. This tells us that we cannot hope to give a Turing-degree based characterization of which combinations of $P^0$, $P^\text{fin}$, and $P^\infty$ give $m$-complete $\Sigma_3$ index sets and which do not.

\subsection*{Acknowledgments} 
The author had a lot of useful input and stimulating conversations during the course of the work, including with Julia Knight, Alexei Myasnikov, Arnold Miller, and Steffen Lempp. Special thanks to Uri Andrews and to Tullia Dymarz for being such supportive and fantastic advisors.

\bibliographystyle{amsalpha}
\bibliography{ref}

\end{document}